\RequirePackage{fix-cm}
\documentclass[smallextended,referee,envcountsect,nospthms]{svjour3}   
\smartqed \usepackage{amsmath, amssymb, latexsym,dsfont}
\usepackage{amsthm}
\usepackage{graphicx}
\usepackage{mathptmx}       
\usepackage{helvet}         
\usepackage{courier}        
\usepackage{type1cm}        
\usepackage{makeidx}         
\usepackage{subcaption}
\usepackage{subeqnarray}
\usepackage{makecell}                             
\usepackage{epstopdf}
\newtheorem{thr}{Theorem}[section]
\newtheorem{co}[thr]{Corollary}
\newtheorem{lm}[thr]{Lemma}
\newtheorem{pr}[thr]{Proposition}
\newtheorem{ex}[thr]{Example}
\newtheorem{defn}[thr]{Definition}
\newtheorem{rem}[thr]{Remark}

\begin{document}
\title{Approximating Properties of Metric and Generalized Metric Projections in Uniformly Convex and Uniformly Smooth Banach Spaces}
\titlerunning{Approximating Properties of Metric and Generalized Metric Projections}        
\author{Akhtar A. Khan, Jinlu Li}
\institute{ Akhtar Khan (Corresponding author)  \at
              School of Mathematical Sciences, Rochester Institute of Technology, Rochester, New York, 14623, USA. \email{aaksma@rit.edu}
                           \and 
             Jinlu Li \at
              Department of Mathematics, Shawnee State University, Portsmouth, Ohio 45662, USA.  \email{jli@shawnee.edu)}
{\large}}
\date{Received: date / Accepted: date}

\maketitle

\begin{abstract}
This note conducts a comparative study of some approximating properties of the metric projection, generalized projection, and generalized metric projection in uniformly convex and uniformly smooth Banach spaces. We prove that the inverse images of the metric projections are closed and convex cones, but they are not necessarily convex. In contrast, inverse images of the generalized projection are closed and convex cones. Furthermore, the inverse images of the generalized metric projection are neither a convex set nor a cone. We also prove that the distance from a point to its projection at a convex set is a weakly lower semicontinuous function for all three notions of projections. We provide illustrating examples to highlight the different behavior of the three projections in Banach spaces. 
\end{abstract}
\section{Introduction}\label{KL3-S1}
The notion of projection onto closed and convex sets has been extensively explored due to its wide-ranging applications. The projection map is equipped with immensely valuable properties in a Hilbert space, making it an indispensable tool in optimization, approximation theory, inverse problems, variational inequalities, image processing, neural networks, machine learning, and others. On the other hand, many critical applications require that projection maps be studied theoretically and computationally in Banach spaces. An important example is the inverse problem of identifying discontinuous parameters where the regularized problem is formulated in a Banach space. As a consequence, many researchers have studied projections in Banach spaces.   Unfortunately, in this setting, metric projection loses many essential features. For an overview of these details and some of the related developments, see \cite{BalGol12,BalMarTei21,Bau03,BorDruChe17,Bou15,Bro13,BroDeu72,Bui02,Bur21,CheGol59,ChiLi05,Den01,DeuLam80,DutShuTho17,GJKS21,Ind14,FitPhe82,KonLiuLiWu22,KroPin13,Li04,Li04a,LiZhaMa08,Nak22,Osh70,Pen05,PenRat98,QiuWan22,Ric16,Sha16,ShaZha17,ZhaZhoLiu19},  and the cited references.

Motivated by the shortcomings of the metric projection in Banach spaces, generalized projection, and generalized metric projection were proposed and used heavily in a Banach space framework, see \cite{Alb93,Alb96}. Although there are exceptions (see \cite{Li04a,Li05}), the two notions are mainly studied in Banach spaces with favorable topological structures, such as uniformly convex and uniformly smooth Banach spaces. The basic properties of the generalized projection and the generalized metric projection and their connections are largely unknown in general Banach spaces. Inspired by this, we recently focused on studying generalized projection and generalized metric projection in the framework of Banach spaces. In \cite{KhaLiRwi22} attempts were made to understand the similarities and differences in the three notions of projections in Banach spaces: the metric projection, the generalized projection, and the generalized metric projection. A comparative study of these projection should help shed some light on their utility and their strengths and weaknesses in various applications.

The main motivation of this research is to strengthen further the understanding of the relationship between the three notions of projections. Surprisingly, it turns out that all three projections exhibit different behaviors in computing inverse images. These results are related to a well-known result given in a Hilbert space setting by Zarantonello~\cite[Lemma~1.5]{Zar71}, see also \cite[Theorem 3.1.3]{GJKS21}. We present several illustrating examples involving the computations of the duality map in finite-dimensional Banach spaces.

The contents of this paper are organized into three sections. After a brief introduction in Section~1, we recall various notions of projections and give new results concerning normalized duality mapping. The main results concerning the approximating properties of various projections are given in Section~3.
\section{Preliminaries}\label{KL3-S2-P}
\subsection{The Normalized Duality Map and new characterizations}
Let $X$ be a real Banach space with norm $\|\cdot\|_X$, let $X^*$ be the topological dual of $X$ with norm $\|\cdot\|_{X^*}$, and let $\langle \cdot,\cdot\rangle$ be the duality pairing between $X^*$ and $X$. Let $B_X$ and $B_{X^*}$ be the closed unit balls in $X$ and $X^*$, respectively. For details on the notions recalled in this section, see \cite{Tak00}.

Given a uniformly convex and uniformly smooth Banach space $X$ with dual space $X^*$, the normalized duality map $J:X\to X^*$ is a single-valued mapping defined by
$$\langle Jx,x\rangle=\|Jx\|_{\text{\tiny{\emph{X}}}^*}\|x\|_{\text{\tiny{\emph{X}}}}=\|x\|_{\text{\tiny{\emph{X}}}}^2=\|Jx\|_{\text{\tiny{\emph{X}}}^*}^2,\quad \text{for any}\ x\in X.$$

We recall that the modulus of smoothness of the Banach space $X$, denoted by $\rho_X$, is defined by
$$\rho_X(t)=\sup\left\{\frac{\|x+y\|_X+\|x-y\|_X}{2}-1:\ x,y\in X,\ \|x\|_X=1,\ \|y\|_X=t\right\},\quad \text{for}\ t>0.$$
The modulus of convexity of $X$ is the function $\delta_X:[0,2]\to [0,1]$ defined by
$$\delta_X(\varepsilon)=\inf\left\{1-\left\|\frac{x+y}{2}\right\|:\ x,y\in B_X,\ \|x-y\|_X\geq \varepsilon\right\},\ \text{for any}\ \varepsilon \in [0,2].$$

\begin{lm}\label{KL3-S2.2-L2.1} Let $X$ be a uniformly convex and uniformly smooth Banach space, and let $X^*$ be the dual of $X$. Then, the normalized map $J_X:X\to X^*$ has the following properties:
\begin{description}
\item[($J_1$)] $J_X:X\to X^*$ is one-to-one, onto, continuous and homogeneous.
\item[($J_2$)] $J_X$ is uniformly continuous on each bounded subset of $X$.
\item[($J_3$)] For any $x,y\in X$, let $R=\max\{\|x\|_X,\|y\|_X\}$ and let $\Gamma_X$ be the Figiel's constant of $X$. Then,
\begin{align*}
\langle J_X x- J_X y,x-y\rangle&\geq \frac{R^2}{2\Gamma_X} \delta_X\left(\frac{\|x-y\|_X}{2R}\right),\\
\|J_X x- J_X y\|_{X^*}&\leq \frac{R^2}{2\Gamma_X\|x-y\|_X}\sigma_X\left(\frac{16\Gamma_X\|x- y\|_X}{R}\right).
\end{align*}
\end{description}
\end{lm}

The following example will be repeatedly used in this work.
\begin{ex}\label{KL3-S2.2-Ex2.2} Let $X=\mathds{R}^3$ be equipped with the $3$-norm $\|\cdot\|_3$ defined for any $z=(z_1,z_2,z_3)\in X,$ by
$$\|z\|_3=\sqrt[3]{|z_1|^2+|z_2|^3+|z_3|^3}.$$ Then $(X,\|\cdot\|_3)$ is a uniformly convex and uniformly smooth Banach space (and is not a Hilbert space). The dual space of $(X,\|\cdot\|_3)$ is $(X^*,\|\cdot\|_{\frac{3}{2}})$ so  that for any $\psi=(\psi_1,\psi_2,\psi_2)^*$, we have
$$\|\psi\|_{\frac32}=\left(|\psi_1|^{\frac{3}{2}}+|\psi_2|^{\frac{3}{2}}+|\psi_3|^{\frac{3}{2}}\right)^{\frac23}.$$

The normalized duality mappings $J$ and $J^*$ satisfy the following conditions. For any $z=(z_1,z_2,z_3)\in X$ with $z\ne 0$, we have
\begin{equation}\label{KL3-S2.2-Ex2.2-E2.1}
Jz=\left(\frac{|z_1|^2\text{sign}(z_1)}{\|z\|_3},\frac{|z_2|^2\text{sign}(z_2)}{\|z\|_3},\frac{|z_3|^2\text{sign}(z_3)}{\|z\|_3}\right).
\end{equation}
Moreover, for any $\psi=(\psi_1,\psi_2,\psi_3)\in X^*$ with $\psi\ne 0$, we have
\begin{equation}\label{KL3-S2.2-Ex2.2-E2.2}
J^*\psi =\left( \frac{|\psi_1|^{\frac32 -1}\text{sign}(\psi_1)}{\left(\|\psi\|_{\frac32}\right)^{\frac32-2}}, \frac{|\psi_2|^{\frac32 -1}\text{sign}(\psi_2)}{\left(\|\psi\|_{\frac32}\right)^{\frac32-2}},\frac{|\psi_3|^{\frac32 -1}\text{sign}(\psi_3)}{\left(\|\psi\|_{\frac32}\right)^{\frac32-2}}\right).
\end{equation}
\end{ex}
We have the following new characteristic of the normalized duality map.
\begin{pr}\label{KL3-S2.2-ND-P2.2} Let $X$ be a uniformly convex and uniformly smooth Banach space. Let $\theta \ne y\in X.$ Then, the set
$$\left\{x\in X:\ \langle J_Xx,y\rangle\geq 0 \right\},$$
is a closed cone with vertex at $\theta$ in $X$. However, in general, it is not convex.
\end{pr}
\begin{proof} Since the normalized duality map $Jx$ is continuous and homogeneous, it follows at once that the set $\left\{x\in X:\ \langle J_Xx,y\rangle\geq 0 \right\}$ is a closed cone with vertex at the origin.
We construct a counterexample to show that the set $\left\{x\in X:\ \langle J_Xx,y\rangle\geq 0 \right\}$ is not convex.

Let $X=\mathds{R}^3$ be as given in Example~\ref{KL3-S2.2-Ex2.2} We take $v=(3,-2,-1)$, $w=(1,-3,2)$, and $y=(25,37,77)$. Then $\|v\|_3=\|w\|_3$. By Example~\ref{KL3-S2.2-Ex2.2}, we have
\begin{align*}
J_Xv&=\left(\frac{9}{\|v\|_3},\frac{-4}{\|v\|_3},\frac{-1}{\|v\|_3}\right),\\
J_Xw&=\left(\frac{1}{\|w\|_3},\frac{-9}{\|w\|_3},\frac{4}{\|w\|_3}\right),
\end{align*}
which gives
\begin{equation}\label{KL3-S2.2-ND-E2.2}
\langle J_Xv,y\rangle=0,\quad \text{and}\quad\langle J_Xw,y\rangle=0.
\end{equation}
We take a convex combination of $v$ and $w$ by
$$g=\frac23v+\frac13w=\left(\frac73,-\frac73,0\right),$$
which gives $\|g\|_3=\frac73\sqrt[3]{2}$. By Example~\ref{KL3-S2.2-Ex2.2}, we calculate
$J_Xg=\frac{7}{3\sqrt[3]{2}}(1,-1,0),$ which yields
\begin{equation}\label{KL3-S2.2-ND-E2.3}
\langle J_Xg,y\rangle=-14\sqrt[3]{4}<0,
\end{equation}
proving that $g\notin \{x\in X:\, \langle  J_Xx,y\rangle \geq 0\},$ and hence $\left\{x\in X:\ \langle J_Xx,y\rangle\geq 0 \right\}$ is not convex.
\end{proof}

A simple extension of the above result is the following variant.
\begin{pr}\label{KL3-S2.2-ND-P2.3} Let $X$ be a uniformly convex and uniformly smooth Banach space. Let $\theta \ne y\in X$. Then the set
$$\left\{x\in X:\ \langle J_Xx,y\rangle\leq 0 \right\},$$
is a closed cone with vertex at $\theta$ in $X$. However, in general, it is not convex.
\end{pr}

By similar arguments used in the proof of Proposition~\ref{KL3-S2.2-ND-P2.2}, we prove the following result.
\begin{pr}\label{KL3-S2.2-ND-P2.4} Let $X$ be a uniformly convex and uniformly smooth Banach space. For $\theta \ne y\in X$ and $\theta \ne \psi\in X^*$, the set
$$\left\{x\in X:\ \langle J_Xx-\psi,y\rangle\geq 0 \right\},$$
is closed. However, in general, it is not convex.
\end{pr}
\begin{proof} By the properties of the normalized duality map, it follows that $\left\{x\in X:\ \langle J_Xx-\psi,y\rangle\geq 0 \right\}$ is closed. To show that it is not convex, we modify the counterexample given in Proposition~\ref{KL3-S2.2-ND-P2.2}.

Let $v=(3,-2,-1)$, $w=(1,-3,2)$ and $y=(25,37,77)$. For $\beta>0$, we take $\psi=(-\beta,-\beta,-\beta)$ with $\beta <\frac{14\sqrt[3]{4}}{139}$. Then,
\begin{align*}
\langle J_Xv-\psi,y\rangle&=\langle -\psi,y\rangle=139\beta>0,\\
\langle J_Xw-\psi,y\rangle&=\langle -\psi,y\rangle=139\beta>0.
\end{align*}
For $g=\frac23 v+\frac13w$, as in Proposition~\ref{KL3-S2.2-ND-P2.2}, we have $\langle J_X g-\psi\rangle=-14\sqrt[3]{4}+139\beta<0,$ which proves that $g \notin \{x\in X: \langle J_Xx-\psi,y\rangle\geq 0 \}$, and hence $\{x\in X: \langle J_Xx-\psi,y\rangle\geq 0 \} $ is not convex.
\end{proof}

A modification of Proposition~\ref{KL3-S2.2-ND-P2.4} proves the following result.
\begin{pr}\label{KL3-S2.2-ND-P2.5} Let $X$ be a uniformly convex and uniformly smooth Banach space. For $\theta \ne y\in X$ and $\theta \ne \psi\in X^*$, the set
$$\left\{x\in X:\ \langle J_Xx-\psi,y\rangle\leq 0 \right\},$$
is closed. However, in general, it is not convex.
\end{pr}

Before our next result, we recall notions of some specific sets in Banach spaces. Given any Banach space $X$, for any  $u,v\in X$ with $u\ne v$, we write
\begin{description}
\item[($a$)] $[v,u]=\{tv+(1-t)u:\ 0\leq t\leq 1\}.$
\item[($b$)] $[v,u\lceil=\{tv+(1-t)u:\ 0\leq t<\infty\}.$
\item[($c$)] $\rceil u,v \lceil=\{tv+(1-t)u:\ \infty < t <\infty \}.$
\end{description}
The set $[v,u]$ is a closed segment  with end points $u$ and $v$. The set $[v,u\lceil$ is a closed ray in $X$ with initial point $v$ with direction $u-v$, which is a closed convex cone with vertex at $v$ and is a special case of cones in $X$.  The set $\rceil u,v \lceil$ is a line in $X$ passing through points $v$ and $u$.

 We have the following result concerning the images of segments, rays, and lines under the normalized duality map.
\begin{pr} \label{KL3-S2.2-ND-P2.7}  Let $X$ be a uniformly convex and uniformly smooth Banach space. Let $u,v\in X$ with $u\ne v.$  If $u$ and $v$ are linearly dependent, that is, $\theta \in \rceil u,v \lceil$, then we have
\begin{description}
\item[($a$)] $J[v,u]=[Jv,Ju]$, which is a closed segment in $X^*$ with end points $Jv$ and $Ju$.
\item[($b$)] $J[v,u\lceil=[Jv,Ju\lceil$, which is a closed ray in $X^*$ with end points $Jv$ and $Ju$.
\item[($c$)] $J \rceil u,v \lceil=\rceil Ju,Jv \lceil$, which is a 1-d subspace in $X^*$ through point $Jv$ and direction $Ju-Jv$.
\end{description}
Furthermore, if $u$ and $v$ are linearly independent, that is, $\theta \notin \rceil u,v \lceil$, then we have
\begin{description}
\item[($d$)] $J[v,u]$ is a closed curve in $X^*$ with end points $Jv$ and $Ju$; it may not be a segment.
\item[($e$)] $J[v,u\lceil$ is a closed curve in $X^*$ with end point $Jv$ and through $Ju$; it may not be a ray.
\item[($f$)] $J\rceil u,v \lceil$ is a closed curve in $X^*$ through $Jv$ and $Ju$; it may not be a line.
\end{description}
\end{pr}
\begin{proof} Suppose that $u$ and $v$ are linearly dependent, then we can assume that there is a real number $a\ne 0$ such that $u=av$. It follows that
$$[v,u]=\{tv+(1-t)u:\ 0\leq t\leq 1\}=\{(t+(1-t)a)v:\ 0\leq t\leq 1\},$$
Furthermore, by the homogeneity of $J$, we have
\begin{align*}
J[v,u]&=\{J(tv+(1-t)u):\ 0\leq t\leq 1\}\\
&=\{(t+(1-t)a)Jv:\ 0\leq t\leq 1\}\\
&=\{tJv+(1-t)J(av):\ 0\leq t\leq 1\}\\
&=\{(tJv+(1-t)Ju:\ 0\leq t\leq 1\}\\
&=[Jv,Ju],
\end{align*}
which completes the proof of (a). Parts (b) and (c) can be proved by analogous arguments.

For (d), we note that since $X$ is a uniformly convex and uniformly smooth Banach space, the duality map $J$ is one-to-one and continuous mapping from $X$ into $X^*$. It follows that $J[v,u]$ is a closed curve in $X^*$ with end points $Jv$ and $Ju$. We construct a counterexample to show that $J[v,u]$ is not a segment.

Let $(X,\|\cdot\|_3)$ be the uniformly convex and uniformly smooth Banach space with dual space $(X^*,\|\cdot\|_{\frac32})$. Take $u=(0,-1,1)$ and $v=(-1,1,0)$ in $X$. Then,we have
\begin{align*}
Ju&=\frac{1}{\sqrt[3]{2}}(0,-1.1)\in X^*\\
Jv&=\frac{1}{\sqrt[3]{2}}(-1,1,0)\in X^*.
\end{align*}
We take a convex combination $\psi$ of $Ju$ and $Jv$ by
$$\psi=\frac14 Ju+\frac34 Jv=\frac{1}{\sqrt[3]{2}}\left(-\frac34,\frac24,\frac14\right)=\frac{1}{4\sqrt[3]{2}}(-3,2,1).$$
Since $Jv$ and $Ju$ are both in $J[v,u]$, to prove that it is not a segment, we need to prove that
\begin{equation}\label{KL3-NE2.3}
J[v,u]\ne [Ju,Jv].
\end{equation}
Since $\psi\in [Jv,Ju]$, it is sufficient to show that $\psi\notin J[v,u]$. Since $J^*$ is homogeneous, we have
$$J^*\psi=\frac{1}{4\sqrt[3]2}J^*(-3,2,1)=\frac{\sqrt[3]{3^{\frac32}+2^{\frac32}+1}}{4\sqrt[3]{2}}\left(-\sqrt{3},\sqrt{2},1\right).$$
Next we prove that $J^*\notin [v,u]$. If possible, assume that $J^*\psi \in [v,u]$. Then, there exists $\beta\in [0,1]$ such that
$$\frac{\sqrt[3]{3^{\frac32}+2^{\frac32}+1}}{4\sqrt[3]{2}}\left(-\sqrt{3},\sqrt{2},1\right)=\beta v+(1-\beta)u=(-\beta,-1+2\beta,1-\beta)$$
which implies
$$\frac{\sqrt[3]{3^{\frac32}+2^{\frac32}+1}}{4\sqrt[3]{2}}\times \sqrt{3}=1-\frac{\sqrt[3]{3^{\frac32}+2^{\frac32}+1}}{4\sqrt[3]{2}}$$
which is a contradiction. Hence we have shown that $J^*\psi \notin [v,u]$

Since $X$ is uniformly convex and uniformly smooth, both $J$ and $J^*$ are one-to-one and onto mappings, which are inverse of each other. Then, we have
$\psi=J(J^*\psi)\notin J[v,u]$ which implies \eqref{KL3-NE2.3}. Parts (e) and (f) follow from (d) immediately.
\end{proof}
\begin{co}\label{KL3-C2.8}  Let $X$ be a uniformly convex and uniformly smooth Banach space  and let $K$ be a closed cone in $X$ with vertex at $v\in X$.
\begin{description}
\item[(i)] If $v=\theta$, then $JK$ is a closed cone in $X^*$ with vertex at $\theta^* =J\theta.$ However, $JK$ is not necessarily convex even if $K$ is convex.
\item[(ii)] If $v\ne \theta$ or $K$ is a ray with $\theta \notin \overset{\leftrightarrow}{K}$, then $JK$ is not a cone. Here $\overset{\leftrightarrow}{K}$ is a line containing the ray $K$.
\end{description}
\end{co}
\begin{proof} Since $X$ is uniformly convex and uniformly smooth, $J$ is continuous and positive homogeneous, and hence $JK$ is a closed cone in $X^*$ with vertex at $\theta =J\theta.$  To show that the convexity of $K$ does not imply the convexity of $JK$, we construct a counterexample. Let $X=\mathds{R}^3$ be as in Example~\ref{KL3-S2.2-Ex2.2}.  We take $\phi=(1,1,1)\in X^*$ and define
\begin{equation}\label{New2.3}
K=\{w\in X:\ \langle \phi,w\rangle=0 \}.
\end{equation}
Then $K$ is a closed subspace of $X$ which is a closed and convex cone in $X$ with vertex $\theta$. We claim that $JK$ is not convex.

We take two points $u,v\in X$ given by $u=(0,-1,1)$ and $v=(-1,1,0).$ Then, $\langle \phi,u\rangle=0 $ and $\langle \phi,v\rangle=0$ From Proposition~\ref{KL3-S2.2-ND-P2.7}, we have
$Jv=\frac{1}{\sqrt[3]{2}}(-1,1,0)$ and $Jv=\frac{1}{\sqrt[3]{2}}(0,-1,1)$. We define
$$\psi=\frac34 Jv+\frac14 Ju=\frac{1}{4\sqrt[3]{2}}(-3,2,1).$$
Then, we have
$$J^*\psi=\frac{\sqrt[3]{3^{\frac32}+2^{\frac32}+1}}{4\sqrt[3]{2}}\left(-\sqrt{3},\sqrt{2},1\right),$$
which implies
$$\langle \phi,J^*\psi\rangle =\frac{\sqrt[3]{3^{\frac32}+2^{\frac32}+1}}{4\sqrt[3]{2}} \left(-\sqrt{3}+\sqrt{2}+1\right)>0$$
and hence $J^*\psi\notin K.$  As before, this implies that $\psi=JJ^*\psi\notin JK$, proving that $JK$ is not convex. The remaining part can be proved as Proposition~\ref{KL3-S2.2-ND-P2.7}.
\end{proof}

\subsection{Projections in Banach Spaces}

Let $X$ be a uniformly convex and uniformly smooth Banach space and let $C$ be a nonempty, closed, and convex subset of $X$. We define a Lyapunov function $V:X^*\times X\to \mathds{R}$ by the formula:
$$V(\psi,x)=\|\psi\|^2_{X^*}-2\langle \psi,x\rangle+\|x\|_X^2,\quad \text{for any}\ \psi\in X^*,\ x\in X.$$

We shall now recall useful notions of projections in Banach spaces.
\begin{defn}
Let $X$ be a uniformly convex and uniformly smooth Banach space, let $X^*$ be the dual of $X$, and let $C$ be a nonempty, closed, and convex subset of $X$.
\begin{description}
\item[1.] The metric projection $P_C:X\to C$ is a single-valued defined by
$$\|x-P_Cx\|_X\leq \|x-z\|_X,\quad \text{for all}\ z\in C.$$
\item[2.] The generalized projection $\pi_C:X^*\to C$ is a single-value map that satisfies
\begin{equation}\label{KL3-S2.5-GMP-E3.3}
V(\psi,\pi_C\psi)=\inf_{y\in C}V(\psi,y),\quad \text{for any}\ \psi\in X^*.
\end{equation}
\item[3.] The generalized metric projection $\Pi_C:X\to C$ is a single-valued map defined by
\begin{align*}\Pi_C x&=\pi_c(J_Xx),\quad \text{for any}\ x\in X,\\
\pi_C(\psi)&=\Pi_C(J_{X^*}\psi),\quad \text{for any}\ \psi\in X^*.
\end{align*}
\end{description}
\end{defn}

The following result collects some of the basic properties of the metric projection defined above.
\begin{pr}\label{KL3-S2.3-NP-P2.6}  Let $X$ be a uniformly convex and uniformly smooth Banach space and let $C$ be a nonempty, closed, and convex subset of $X$.
\begin{description}
\item[1.] The metric projection $P_C:X\to C$ is a continuous map that enjoys the following variational characterization:
\begin{equation}\label{KL3-S2.3-NP-P2.6-E1}
u=P_C(x)\quad \Leftrightarrow \quad \langle J_X(x-u),u-z\rangle\geq 0,\quad \text{for all}\ z\in C.
\end{equation}
\item[2.] The generalized projection $\pi_C:X^*\to C$ enjoys the following variational characterization: For any $\psi\in X^*$ and $y\in C,$
\begin{equation}\label{KL3-S2.4-GP-E2.5}
y=\pi_C(\psi),\quad \text{if and only if},\quad \langle \psi-J_Xy,y-z\rangle\geq 0,\quad \text{for all}\ z\in C.
\end{equation}
\item[3.] The generalized metric projection $\pi_C:X^*\to C$ enjoys the following variational characterization: For any $\psi\in X^*$ and $y\in C,$
\begin{equation}\label{KL3-S2.4-GP-E2.5-N}
y=\Pi_C(x),\quad \text{if and only if},\quad \langle J_X x-J_X\Pi_Cx,\Pi_Cx-z\rangle\geq 0,\quad \text{for all}\ z\in C.
\end{equation}
\end{description}
\end{pr}

To be specific, we show that a certain inverse image of the metric projection is a closed and convex cone, but it is not necessarily convex (Theorem~\ref{KL3-S3.1-T3.1}). In contrast, inverse images of the generalized projection are closed and convex cones (Theorem~\ref{KL3-S3.1-T3.3}). On the other hand, the inverse images of the generalized metric projection are neither a convex set nor a cone (Theorem~\ref{KL3-S3.1-T3.5}). We also prove, for all three notions of projections, that the distance from a point to its projection at a convex set is
a weakly lower semicontinuous function.
\section{Approximating Properties of the Projections}\label{KL3-S3-AP}
\subsection{Approximating Properties of the Metric projection}\label{KL3-S3.1 APMP}
\begin{thr}\label{KL3-S3.1-T3.1}  Let $X$ be a uniformly convex and uniformly smooth Banach space and let $C$ a nonempty, closed, and convex subset of $X$. For any $y\in C,$ let $x\in X\backslash C$ be such that $y=P_C{x}$. We define the inverse image of $y$ under the metric projection  $P_C:X\to C$ by
$$P_{C}^{-1}(y)=\{u\in X:\ P_{C}(u)=y\}.$$
Then $P_{C}^{-1}(y)$ is a closed cone with vertex at $y$ in $X$. However, $P_{C}^{-1}(y)$ is not convex, in general.
\end{thr}
\begin{proof} Since $P_{C}$ is continuous, $P_C^{-1}$ is closed. To show that $P_C^{-1}(y)$ is a cone with vertex at $y$, note that for any $u\in P_C^{-1}(y)$ with $u\ne y$, and for any $t\geq 0$, by the fact $J_X$ is homogeneous, we have
$$\langle J_X(y+t(u-y)-y),y-z\rangle=t\langle J_X(u-y),y-z\rangle\geq 0,\ \text{for all}\ z\in C. $$
Thus, by \eqref{KL3-S2.3-NP-P2.6-E1}, we get $y+t(u-y)\in P_C^{-1}(y)$, ensuring that $P_C^{-1}(y)$ is a cone with vertex at $y$.

We will construct a counter-example to show that the set $P_C^{-1}(y)$ is not convex. Let $X=\mathds{R}^3$ be the uniformly convex and uniformly smooth Banach space equipped with the $\|\cdot\|_3$-norm as defined in Proposition~\ref{KL3-S2.2-ND-P2.2}. We take $y=(25,37,77)\in X$. We define the closed segment in $X$ with $\theta$ and $y$ as the end points by
$$C=\{ty\in X:\ t\in [0,1]\}.$$
For $v=(3,-2,-1)$, $w=(1,-3,2)$, we define
$$x=v+y=(28,35,76)\quad \text{and}\quad z=w+y=(26,34,79).$$
Then, as in Proposition~\ref{KL3-S2.2-ND-P2.2}, for any $ty\in C$ with $t\in [0,1]$, we have
$$\langle J_X(x-y),y-ty\rangle=(1-t)\langle J_Xv,y\rangle=0,\quad \text{for any}\ ty\in C.  $$
By \eqref{KL3-S2.3-NP-P2.6-E1}, this implies that $x\in P_C^{-1}(y)$. We can similarly prove that $z\in P_C^{-1}(y)$. Next, we define
\begin{align*}g&=\frac{2}{3}v+\frac13 w=\left(\frac73,-\frac73,0\right),\\
h&=\frac23x+\frac13 z=\frac23 v+\frac23 w+y=g+y,
\end{align*}
and compute for $ty\in C$ with $t\in [0,1]:$
$$\langle J_X(h-y),y-ty\rangle=\langle J_X\left(\frac23 x+\frac13z-y\right),y-ty\rangle=\langle J_X(g),y-ty\rangle=(1-t)\langle J_Xg,y\rangle=-14\sqrt[3]{4}(1-t)<0.$$
Again, by \eqref{KL3-S2.3-NP-P2.6-E1}, we have $P_Ch\neq y$, that is $h\notin P_C^{-1}(y)$, and hence it is not convex.
\end{proof}
\begin{rem}  Theorem~\ref{KL3-S3.1-T3.1} can be proved without using the basic variational principle of the metric projection.
\end{rem}
\begin{thr}\label{KL3-S3.1-T3.2}  Let $X$ be a uniformly convex and uniformly smooth Banach space and let $C$ be a nonempty, closed, and convex subset of $X$. Then the distance from a point to its projection at $C$ is a weakly lower semicontinuous function. That is, for any $\{x_n\}\subset X$ and $x\in X$, we have
$$x_n\rightharpoonup x\quad \Rightarrow\quad \|x-P_Cx\|_X\leq \liminf_{n\to \infty}\|x_n-P_cx_n\|_X.$$
\end{thr}
\begin{proof} Since the claim is trivial for $x\in C,$ we assume that $x\notin C.$ For every $n\in \mathds{N}$, since $P_Cx_n\in C,$ by the variational characterization \eqref{KL3-S2.3-NP-P2.6-E1} of the metric projection, we have
$$\langle J_X(x-P_Cx),P_Cx-P_Cx_n\rangle\geq 0, $$
which can be rearranged as
$$\langle J_X(x-P_Cx),x_n-P_Cx_n\rangle\geq \langle J_X(x-P_Cx),x-P_Cx\rangle+\langle J_X(x-P_Cx),x_n-x\rangle,$$
and subsequently
$$\|J_X(x-P_Cx)\|_{X^*}\|x_n-P_Cx_n\|\geq \|x-P_Cx)\|^2+\langle J_X(x-P_Cx),x_n-x\rangle.$$
Since $x\ne P_Cx$, we obtain
$$\|J_X(x-P_Cx)\|_{X^*}=\|x-P_Cx\|_X>0,$$
which implies
$$\|x_n-P_Cx_n\|\geq \|x-P_Cx\|_X+\frac{\langle J_X(x-P_Cx),x_n-x\rangle }{\|x-P_Cx\|_X},\quad n\in \mathds{N}.$$
Taking $\liminf$ in the above inequality, yields the desired result. Note that here we used the fact that for $J_X(x-P_Cx)\in X^*,$ we have
$\langle J_X(x-P_Cx),x_n-x\rangle\to 0\ \text{as}\ n\to \infty.$ The proof is complete.
\end{proof}
\subsection{Approximating Properties of the Generalized projection}\label{KL3-S3.2 APGP}
\begin{thr}\label{KL3-S3.1-T3.3} Let $X$ be a uniformly convex and uniformly smooth Banach space and let $C$ be a nonempty, closed, and convex subset of $X$. For any $y\in C$, and any $\psi\in X^*$ with $\psi\ne J_Xy$ such that $y=\pi_C(\psi)$, we define the inverse image of $y$ under the generalized projection $\pi_C$ in $X^*$ by
$$\pi_C^{-1}(y)=\{\phi\in X^*:\ \pi_C(\phi)=y\}.$$
Then $\pi_C^{-1}(y)$ is a $\|\cdot\|_{X^*}$-closed and convex cone with vertex at $J_Xy$ in $X^*$.
\end{thr}
\begin{proof} Let $\psi,\phi\in \pi_C^{-1}(y)$ and $\alpha\in [0,1]$ be arbitrary. Then
$$\langle (\alpha \psi+(1-\alpha)\phi)-J_Xy,y-z\rangle= \alpha \langle \psi-J_Xy,y-z\rangle+(1-\alpha)\langle \phi-J_Xy,y-z\rangle\geq 0,\quad \text{for all}\ z\in C,$$
which due to the variational characterization \eqref{KL3-S2.4-GP-E2.5} implies that $\pi_{C}(\alpha \psi+(1-\alpha)\phi)=y$ and hence $(\alpha \psi+(1-\alpha)\phi)\in \pi_C^{-1}(y)$. Thus, proving the convexity of $\pi_C^{-1}(y)$.

Moreover, since for any $\psi\in \pi_C^{-1}(y)$ and for any $t\geq 0,$ we have
$$\langle (J_Xy+t(\psi-J_Xy))-J_Xy,y-z\rangle=t\langle \psi-J_Xy,y-z\rangle\geq 0, \quad \text{for all}\ z\in C,$$
by appealing to \eqref{KL3-S2.4-GP-E2.5} once again, we obtain $\pi_C(J_Xy+t(\psi-J_Xy))=y$, and hence $J_Xy+t(\psi-J_Xy)\in \pi_C^{-1}(y),$ proving $\pi_C^{-1}(y)$ is a cone with vertex at $J_Xy$ in $X^*$.

Finally, we prove that $\pi_C^{-1}(y)$ is $\|\cdot\|_{X^*}$-closed in $X^*$. Let $\{\psi_n\}\subset \pi_C^{-1}(y)$ be a sequence converging to $\psi.$ We note that for an arbitrary and fixed $z\in C$ and for a fixed $y$, we have
$$|\langle \psi-J_Xy,y-z\rangle-\langle \psi_n-J_Xy,y-z\rangle |\leq \|\psi-\psi_n\|_{X^*}\|y-z\|_X\to 0\ \text{as}\ n\to \infty,$$
and since $\langle \psi_n-J_Xy,y-z\rangle\geq 0 $ for $z\in C$, we infer that
$$\langle \psi-J_Xy,y-z\rangle\geq 0,\quad \text{for any}\ z\in C,$$
which implies that $\pi_C(\psi)=y$  and hence $\psi\in \pi_C^{-1}(y)$, proving that $\pi_C^{-1}(y)$ is indeed closed.
\end{proof}

If $X$ is a uniformly convex and uniformly smooth Banach space and $C\subset X$ is nonempty, closed, and convex, then the mapping $\pi_C:X^*\to C$ is continuous. We have the following stronger result:
\begin{thr}\label{KL3-S3.1-T3.4}  Let $X$ be a uniformly convex and uniformly smooth Banach space and let $C$ be a nonempty, closed, and convex subset of $X$. Let $\{\psi_n\}\subset X^*$, $\psi\in X^*$, and $y\in C$. Assume that the following conditions are satisfied:
\begin{description}
\item[(a)] $\psi_n \rightharpoonup \psi$ weak$^*$ as $n\to \infty$.
\item[(b)] $\{\psi_n\}$ is $\|\cdot\|_{X^*}$ bounded.
\item[(c)] $\displaystyle  \lim_{n\to \infty}\pi_C(\psi_n)=y.$
\end{description}
Then $y=\pi_C(\psi).$
\end{thr}
\begin{proof} By the variational characterization \eqref{KL3-S2.4-GP-E2.5}, for each $n\in \mathds{N}$, we have
\begin{equation}\label{KL3-S2.1-NDM-E3.3}
\langle \psi_n-J_X(\pi_C(\psi_n)),\pi_C(\psi_n)-z\rangle \geq 0.
\end{equation}
For any arbitrary $z\in C$, we have
\begin{align}
|\langle &\psi_n-J_X(\pi_C(\psi_n)),\pi_C(\psi_n)-z\rangle-\langle \psi-J_X(y),y-z\rangle|\notag\\
&\leq |\langle \psi_n-J_X(\pi_C(\psi_n)),\pi_C(\psi_n)-z\rangle-\langle \psi_n-J_X(\pi_C\psi_n),y-z\rangle|\notag\\
&+|\langle \psi_n-J_X(\pi_C(\psi_n)),y-z\rangle-\langle \psi-J_X(y),y-z\rangle|\notag\\
&=|\langle \psi_n-J_X(\pi_C(\psi_n)),\pi_C(\psi_n)-y\rangle |+|+\langle \psi_n-J_X(\pi_C(\psi_n)-(\psi-J_X(y))),y-z\rangle|\notag\\
&\leq \|\psi_n-J_X(\pi_C(\psi_n))\|_{X^*}\|\pi_C(\psi_n)-y \|_X+|+\langle \psi_n-J_X(\pi_C(\psi_n)-(\psi-J_X(y))),y-z\rangle|.\label{KL3-S2.1-NDM-E3.4}
\end{align}
The imposed conditions imply that
$\|\psi_n-J_X(\pi_C(\psi_n))\|_{X^*}\|\pi_C(\psi_n)-y \|_X\to 0$
and by the continuity of $J_X$, we have $\lim_{n\to \infty} J_X(\pi_C(\psi_n))=J_Xy,$ and hence
\begin{equation}\label{KL3-S2.1-NDM-E3.5}
(\psi_n-J_X(\pi_C(\psi_n)))\rightharpoonup (\psi-J_Xy),\quad \text{weak}^*,\ \text{as}\ n\to \infty.
\end{equation}
Then, for any fixed $z\in C,$ we have
\begin{equation}\label{KL3-S2.1-NDM-E3.6}
\langle \psi_n-J_X(\pi_C(\psi_n))-(\psi-J_Xy),y-z\rangle \to 0.
\end{equation}
Combining \eqref{KL3-S2.1-NDM-E3.3}, \eqref{KL3-S2.1-NDM-E3.4}, \eqref{KL3-S2.1-NDM-E3.5}, and \eqref{KL3-S2.1-NDM-E3.6} it follows that
$$|\langle \psi_n-J_X(\pi_C(\psi_n)-(\psi-J_X(y)))),y-z\rangle|\to 0,$$
and hence
$$\langle \psi-J_Xy,y-z\rangle\geq 0 ,\quad \text{for arbitrary} \ z\in C.$$
It follows by the variational characterization \eqref{KL3-S2.4-GP-E2.5} that $y=\pi_C(\psi).$
\end{proof}
\subsection{Approximating Properties of the Generalized Metric projection}\label{KL3-S3.3 APGMP}
\begin{thr}\label{KL3-S3.1-T3.5} Let $X$ be a uniformly convex and uniformly smooth Banach space and let $C$ be a nonempty, closed, and convex subset of $X$. For any $y\in C$, and $x\in X$ with $x\ne y$ and $y=\Pi_C(x)$, we define the inverse image of $y$ under the generalized metric projection $\Pi_C$ in $X$ by
$$\Pi_C^{-1}(y)=\{u\in X|\ \Pi_Cu=y\}.$$
Then, we have
$$\Pi_C^{-1}(y)=J_{X^*}(\pi_C^{-1}y).$$
Furthermore, in general, $\Pi_C^{-1}(y)$ is nether a convex set nor a cone.
\end{thr}
\begin{proof} Suppose $y\in C$ and $x\in X$ with $x\ne y$ such that $y=\Pi_C(x)$. By the definition of $\Pi_C:X\to C$, we have
$$\pi_C(J_X x)=\Pi_C x=y.$$
Since $J_X$ is a one-to-one, onto, and  single-valued, from $x\ne y$, it follows that $J_Xx\ne J_Xy$. Recall that $\pi_C^{-1}y$ is a closed and convex cone with vertex at $J_Xy$ in $X^*$. Notice that for any $\psi\in X^*$, we have $\pi_C\psi=\Pi_C(J_{X^*}\psi)$. It follows that
$$\Pi_C^{-1}y=\{x\in X: \ \Pi_Cx=y\}=\{J_{X^*}\psi \in X: \psi \in X^*\ \text{with}\ \Pi_C(J_{X^*}\psi)=\pi_C\psi=y\}=J_{X^*}(\pi_C^{-1}y).$$
For any $y\in C$ and $x\in X$ with $x\ne y$, we have $y=\Pi_C x$, if and only if, $y=\pi_C(J_Xx)$. That is, $x\in \Pi_C^{-1}y$, if and only if, $J_Xx=\pi_C^{-1}y.$

Since $J_{X^*}=J_X^{-1},$ it follows that
$$\pi_C^{-1}(y)=J_X(\Pi_C^{-1}y),\quad \text{and}\ \Pi_C^{-1}y=J_{X^*}(\pi_C^{-1}y).$$
Following Theorem~\ref{KL3-S3.1-T3.3}, the proof of the first claim is complete.

We construct a counter-example to show that $\Pi_C^{-1}(y)$ is not convex. Let $X=\mathds{R}^3$ be the uniformly convex and uniformly smooth Banach space equipped with the $\|\cdot \|_3$ norm. We take $y=\left(\frac{1}{\sqrt[3]{3}},\frac{1}{\sqrt[3]{3}},\frac{1}{\sqrt[3]{3}}\right)\in X.$ Then $\|y\|_3=1$, and hence $\langle J_Xy,y\rangle=1. $ As before, we define a convex subset $C$ by $C=\{ty\in X:\ t\in [0,1]\}$.

Let $v=(1.66, 1,-1)$, $w=(-1,1,1.66)$. We calculate
$$J_Xv=\frac{(1.66^2,1,-1)}{\sqrt[3]{1.66^3+1+1}}=\frac{(2.7556,1,-1)}{\sqrt[3]{6.574296}},$$
which gives $\langle J_xv,y\rangle >1.$

Then, using $\langle J_xv,y\rangle >1$ and $\langle J_xy,y\rangle =1,$ we have $\langle J_Xv-J_Xy,y\rangle>0 $, which implies that
$$\langle J_Xv-J_Xy,y-ty\rangle=(1-t)\langle J_Xv-J_Xy,y\rangle\geq 0,\quad \text{for any}\ ty\in C,\ t\in [0,1].  $$
The above inequality, due to \eqref{KL3-S2.4-GP-E2.5} implies that $v\in \Pi_C^{-1}(y)$. Analogously, $w\in \Pi_C^{-1}(y).$

We take $h=\frac12v+\frac12w=(0.33,1,0.33)$. Since,
$$J_hh=\frac{(0.33^2,1,0.33^2)}{\sqrt[3]{0.33^3+1+0.33^3}},$$
we obtain
\begin{equation}\label{KL3-S2.1-NDM-E3.10}
\langle J_Xh,y\rangle=\left\langle \frac{(0.33^2,1,0.33^2)}{\sqrt[3]{1.071874}},\left(\frac{1}{\sqrt[3]{3}},\frac{1}{\sqrt[3]{3}},\frac{1}{\sqrt[3]{3}}\right) \right\rangle=\frac{1.2178}{\sqrt[3]{3.215622}}<1.
\end{equation}
By the above equation and $\langle J_xy,y\rangle =1,$ we have
\begin{equation}\label{KL3-S2.1-NDM-E3.11}
\langle J_Xh-J_Xy,y\rangle=\frac{1.2178}{\sqrt[3]{3.215622}}-1<0.
\end{equation}
Furthermore, the above estimates, for any $ty\in C$ with $t\in [0,1)$ yields
$$\langle J_Xh-J_Xy,y-ty\rangle=(1-t)\langle J_Xh,y\rangle=(1-t)\left(\frac{1.2178}{\sqrt[3]{3.215622}}-1\right)<0.$$
Using \eqref{KL3-S2.4-GP-E2.5}, we deduce that $\Pi_Ch\ne y,$ that is, $h\notin \Pi_C^{-1}(y)$, and hence $\Pi_C^{-1}(y)$it is not convex.

Finally, since $y\in \Pi_C^{-1}(y)\cap C$, it suffices to prove that $\Pi_C^{-1}(y)$ is not a cone with vertex $y$.

We take
$$y=\left(\frac{1}{\sqrt[3]{3}},\frac{1}{\sqrt[3]{3}},\frac{1}{\sqrt[3]{3}}\right)\in X$$ and $C=\left\{ty\in X:\ t\in [0,1]\right\}$. Let
$$u=\left(\frac{2}{\sqrt[3]{3}},-\frac{1}{\sqrt[3]{3}},\frac{1}{\sqrt[3]{3}}\right)\in X.$$
Then
\begin{equation}\label{KL3-S2.1-NDM-E3.12}
\langle J_Xu,y\rangle= \frac{\frac43}{\sqrt[3]{\frac{10}{3}}} >1.
\end{equation}
Furthermore,
$$\langle J_Xu-J_Xy,y\rangle=\frac{\frac43}{\sqrt[3]{\frac{10}{3}}}-1>0,$$
which implies that
\begin{equation}\label{KL3-S2.1-NDM-E3.13}
\langle J_Xu-J_Xy,y-ty\rangle =(1-t)\langle J_Xu-J_Xy,y\rangle\geq 0,\quad \text{for any}\ ty\in C,\ \text{where}\ t\in [0,1].
\end{equation}
The above inequality, due to \eqref{KL3-S2.4-GP-E2.5} implies that $u\in \Pi_C^{-1}(y)$. Now, let
$$g=\frac12(u-y)+y=\frac12 u+\frac12 y=\left(\frac{3}{2\sqrt[3]{3}},0,\frac{1}{\sqrt[3]{3}}\right),$$
which results in
\begin{equation}\label{KL3-S2.1-NDM-E3.14}
\langle J_Xg,y\rangle =\frac{\frac{13}{6}}{\sqrt[3]{\frac{35}{3}}}<1.
\end{equation}
Then,
$$\langle J_Xg-J_Xy,y\rangle <0.$$
Combining the above equations, we have
$$\langle J_Xg-J_Xy,y-ty \rangle=(1-t)\langle J_Xg-J_Xy,y\rangle<0,\quad \text{for any}\ ty\in C,\ t\in [0,1).$$
By \eqref{KL3-S2.4-GP-E2.5}, we have $g\not\in \Pi_C^{-1}(y)$. This proves that $\Pi_C^{-1}(y)$ is not a cone with vertex at $y$.
\end{proof}

If $X$ is a uniformly convex and uniformly smooth Banach space and $C\subset X$ is nonempty, closed, and convex, then the mapping $\Pi_C:X\to C$ is continuous. Next, we prove a stronger result.
\begin{thr}\label{KL3-S3.1-T3.6} Let $X$ be a uniformly convex and uniformly smooth Banach space and let $C$ be a nonempty, closed, and convex subset of $X$. Let $\{x_n\}\subset X$, $x\in X$, and $y\in C$. Assume that the following conditions are satisfied:
\begin{description}
\item[(a)] $J_Xx_n \rightharpoonup J_X x,$ weakly $^*$, as $n\to \infty.$
\item[(b)] $\{x_n\}$ is $\|\cdot\|_{X}$ bounded.
\item[(c)] $\displaystyle \lim_{n\to \infty}\Pi_C(x_n)=y.$
\end{description}
Then $y=\Pi_C(x).$
\end{thr}
\begin{proof} By \eqref{KL3-S2.4-GP-E2.5}, we have
$$\langle J_Xx_n-J_X(\Pi_C(x_n)),\Pi_C(x_n)-z\rangle\geq 0,\quad \text{for all}\ z\in C. $$
Similar to the arguments used to above, for any $z\in C,$ we obtain
$$\langle J_Xx-J_Xy,y-z\rangle \geq 0,$$
which implies that  $y=\Pi_C(x).$
\end{proof}

We recall that we denote the modulus of convexity and the modulus of smoothness of a Banach space  $X$ by $\delta_X$ and $\rho_X$.
\begin{thr}\label{KL3-S3.1-T3.8} Let $X$ be a uniformly convex and uniformly smooth Banach space and let $C$ be a nonempty, closed, and convex subset of $X$. Let $\{x_n\}\subset X$ and let $x\in X.$ Let $R=\max\{\|x\|_X,\|\Pi_Cx\|_X\}$. Then there is a number $K_R\in (0,1]$ such that
$$x_n \rightharpoonup x,\quad \Rightarrow K_R\|x-\Pi_Cx\|_X\leq \liminf_{n\to \infty}\|x_n-\Pi_Cx_n\|_X.$$
\end{thr}
\begin{proof} Since the proof is trivial for $x=\Pi_Cx$, we assume that $x\ne \Pi_Cx.$ For every $n\in \mathds{N}$, due to the variational characterization \eqref{KL3-S2.4-GP-E2.5}, we have
$$\langle J_Xx-J_X\Pi_Cx,\Pi_Cx-\Pi_Cx_n\rangle\geq 0, $$
which can be rearranged as follows
$$\langle J_Xx-J_X\Pi_Cx,x_n-\Pi_Cx_n\rangle\geq \langle J_Xx-J_X\Pi_Cx,x-\Pi_Cx\rangle+\langle J_Xx-J_X\Pi_Cx,x_n-x\rangle. $$
Let $R$ be as given and let $\Gamma_X\in (1,1.7)$ be the Figiel's constant. Then, we have
\begin{align*}
\langle J_Xx-J_X\Pi_Cx,x-\Pi_Cx\rangle &\geq \frac{R^2}{2\Gamma_X}\delta_x\left(\frac{\|x-\Pi_Cx\|_X}{2R}\right),\\
\|J_Xx-J_X\Pi_Cx\|_{X^*} &\leq \frac{R^2}{2\Gamma_X\|x-\Pi_Cx\|_X}\rho_X\left(\frac{16\Gamma_X\|x-\Pi_Cx\|_X}{R}\right).
\end{align*}
Then,
\begin{align*}
\frac{R^2}{2\Gamma_X\|x-\Pi_Cx\|_X}\rho_X&\left(\frac{16\Gamma_X\|x-\Pi_Cx\|_X}{R}\right) \|x_n-\Pi_Cx_n\|_X \\
&\geq \|J_Xx-J_X\Pi_Cx\|_{X^*}  \|x_n-\Pi_Cx_n\|_X\\
&\geq \langle J_Xx-J_X\Pi_Cx,x_n-\Pi_Cx_n\rangle \\
&\geq \langle J_Xx-J_X\Pi_Cx,x-\Pi_Cx\rangle +\langle J_X x-J_X \Pi_Cx,x_n-x\rangle\\
& \geq \frac{R^2}{2\Gamma_X}\delta_x\left(\frac{\|x-\Pi_Cx\|_X}{2R}\right)+ \langle J_X x-J_X \Pi_Cx,x_n-x\rangle,
\end{align*}
which can be written as
$$\|x_n-\Pi_Cx_n\|_X\geq \frac{\delta_X\left(\frac{\|x-\Pi_Cx\|}{2R}\right)}{\rho_X\left(\frac{16\Gamma_X\|x-\Pi_Cx\|_X}{R}\right)}\|x-\Pi_Cx\|_X+\frac{ \langle J_X x-J_X \Pi_Cx,x_n-x\rangle}{\frac{R^2}{2\Gamma_X\|x-\Pi_Cx\|_X}\rho_X\left(\frac{16\Gamma_X\|x-\Pi_Cx\|_X}{R}\right)}.$$
We define a positive number $K_R$ by
$$K_R:=\frac{\delta_X\left(\frac{\|x-\Pi_Cx\|}{2R}\right)}{\rho_X\left(\frac{16\Gamma_X\|x-\Pi_Cx\|_X}{R}\right)}.$$
Since $x\ne \Pi_Cx,$ we have
\begin{equation}\label{KL3-S2.1-NDM-E3.17}
\|x_n-\Pi_Cx_n\|_X\geq K_R\|x-\Pi_Cx\|+\frac{ \langle J_X x-J_X \Pi_Cx,x_n-x\rangle}{\frac{R^2}{2\Gamma_X\|x-\Pi_Cx\|_X}\rho_X\left(\frac{16\Gamma_X\|x-\Pi_Cx\|_X}{R}\right)}.
\end{equation}
For the fixed $J_Xx-J_X\Pi_Cx\in X^*,$ we have
\begin{equation}\label{KL3-S2.1-NDM-E3.18}
\langle J_X(x-P_Cx).x_n-x\rangle\to 0.
\end{equation}
However, since $K_R\leq 1,$ letting $\liminf$ in \eqref{KL3-S2.1-NDM-E3.17} and using \eqref{KL3-S2.1-NDM-E3.18}, we get the desired result.
\end{proof}

\bibliographystyle{spmpsci_unsrt}
\bibliography{BIB-PM}
\end{document}